\documentclass[a4paper,12pt]{article}
\usepackage{a4wide}
\usepackage{amsmath}
\usepackage{amssymb}
\usepackage{amsthm}
\usepackage{latexsym}
\usepackage{graphicx}
\usepackage[english]{babel}
\usepackage{makeidx}

\newtheorem{exm} [subsection]{Example}

\newtheorem{prop}[subsection]{Proposition}

\newtheorem{teor}[subsection]{Theorem}
\newtheorem{lema}[subsection]{Lemma}

\newcommand{\Nn}{\mathbb N^n}
\newcommand{\Nni}{\mathbb N_{\infty}^n}
\newcommand{\MG}{\mathcal M(\Gamma)}

\begin{document}
\selectlanguage{english}
\frenchspacing

\large
\begin{center}
\textbf{Vertex cover algebras of simplicial multicomplexes}

Mircea Cimpoea\c s
\end{center}
\normalsize

\begin{abstract}
We define vertex cover algebras for weighted simplicial multicomplexes and prove basics properties of them. Also,
we describe these algebras for multicomplexes which have only one maximal facet and we prove that they are finitely generated.


\end{abstract}

\section*{Introduction}

\hspace{12pt}
We denote by $\mathbb N$ the set of nonnegative integers. Let $a,b\in \Nn$ be two vectors. We say that $a\leq b$ if
$a(i)\leq b(i)$ for all $i\in [1,n]$, where $a=(a(1),\ldots,a(n))$ and $b=(b(1),\ldots,b(n))$.

Stanley \cite{stan} calls a subset $\Gamma\subset \Nn$ a \emph{multicomplex} if for all $a\in \Gamma$ and all $b\in\Nn$ with
$b\leq a$, it follows that $b\in \Gamma$. Herzog and Popescu in \cite{hp} extended this definition as follows. Denote $\mathbb N_{\infty} =\mathbb N \cup \{\infty\}$ and set $k\leq \infty$ for all $k\in \mathbb N$. A subset $\Gamma\subset \Nni$ is called a \emph{multicomplex}, if the following two conditions hold:

(1) for all $a\in \Gamma$ and all $b\in\Nni$ with $b\leq a$, it follows that $b\in \Gamma$.

(2) for each $a\in \Gamma$ there exists $m\in\MG$ with $a\leq m$, where $\MG$ is the set of maximal elements in $\Gamma$, with respect to $\leq$.

The elements of $\Gamma$ are called \emph{facets} and the elements of $\MG$ are called \emph{maximal facets}. One can easily see that $\MG$ is a nonempty finite subset of $\Gamma$. Moreover, $\Gamma = \{a\in\Nni\;:\; a\leq m $ for some $ m\in\MG \}$.

Let $K$ be a field and $S=K[x_1,\ldots,x_n]$ the polynomial ring over $K$. Let $\Gamma$ be a multicomplex, and let $I(\Gamma)$ be the $K-subspace$ in $S$ spanned by all monomials $x^a:=x_1^{a(1)}\cdots x_n^{a(n)}$ such that $a\notin \Gamma$. Note that $I(\Gamma)$ is a monomial ideal. Conversely, given an arbitrary monomial ideal $I\subset S$, there is a unique multicomplex $\Gamma$ with $I=I(\Gamma)$. According to \cite[Corollary 9.8]{hp}, $\Gamma$ is the unique smallest multicomplex containing the set $A=\{a:\;x^{a}\notin I\}$.

Herzog, Hibi and Trung introduced in \cite{hht} the notion of vertex cover algebras for weighted simplicial complexes. In this paper, we extend that concept for simplicial multicomplexes. Herzog, Hibi and Trung proved that the vertex cover algebras are finitely generated, but this is not the case, in general, for multicomplexes, as Example $1.1$ shows. We describe the vertex cover algebras for multicomplexes which have only one maximal facet. More precisely, we show how we can reduce to the case when the maximal facet is a vector in $\Nn$ with all entries nonzero, see Theorem $1.6$. Also, we prove that these algebras are finitely generated.

\textbf{Aknowledgement.} The author owes a special thank to Dr. Mihai Cipu for valuables discussions regarding the second part of Theorem $1.6$.

\footnotetext[1]{This work was supported by a grant of the Romanian National Authority for Scientific
Research, CNCS – UEFISCDI, project number PN-II-ID-PCE-2011-3-1023.}

\newpage
\section{Vertex cover algebras of simplicial multicomplexes}

\hspace{12pt}
Let $\Gamma\subset \Nni$ be a multicomplex and let $\MG$ be the set of maximal facets of $\Gamma$. Consider a function
\[ \omega: \MG \rightarrow \mathbb N\setminus \{0\},\;\; m\mapsto \omega_{m} \]
that assigns to each maximal facet a positive integer. In this case, $\Gamma$ is called a \emph{weighted multicomplex}, denoted by $(\Gamma,\omega)$. We call $a\in\Nn$ a \emph{vertex cover} of $(\Gamma,\omega)$ of order $k$ if:
\[ \sum_{i=1}^{n} a(i)m(i) \geq k\omega_m,\;\;for\;\;all\;\; m\in\MG,\]
where we define $0\cdot \infty := 0$. The canonical weight function on a multicomplex $\Gamma$ is the weight function $\omega_0(m)=1$ for all maximal facets $m\in \MG$.

Let $S[t]$ be a polynomial ring over $S$ in the indeterminate $t$, and consider the $K$-vector space $A_k(\Gamma,\omega) \subset S[t]$ generated by all monomials $x_1^{a(1)} \cdots x_n^{a(n)} t^k$ such that $a\in \Nn$ is a vertex cover of $\Gamma$ of order $k$. We define
\[ A(\Gamma,\omega) = \bigoplus_{k\geq 0} A_k(\Gamma,\omega)\;\;with\;\; A_0(\Gamma,\omega)=S. \]

If $a$ is a vector cover of order $k$, and $b$ is a vector cover of order $l$, one can easily see that $a+b$ is a vertex cover of order $k+l$. This implies that $A_k(\Gamma,\omega)\cdot A_l(\Gamma,\omega) \subset A_{k+l}(\Gamma,\omega)$ and therefore $A(\Gamma,\omega)$ is a graded $S$-algebra. We call it the \emph{vertex cover algebra} of the weighted simplicial multicomplex $(\Gamma,\omega)$. For simplicity, we will use the notation $A(\Gamma)$
for $A(\Gamma,\omega_0)$.

\begin{exm}
Let $\Gamma=\{a\;:\;a\leq (0,\infty)\;or\; a\leq (2,0) \} \subset  \mathbb N_{\infty}^{2}$. The set of maximal facets of $\Gamma$ is $\MG=\{ (0,\infty),(2,0) \}$. We consider the canonical weight function on $\Gamma$. According to the definition, a vector $a\in\mathbb N^2$ is a vertex cover of order $k\geq 1$, if and only if $2a(1)\geq k$ and 
$a(2)\geq 1$. Therefore,
$ A(\Gamma) = K[x_1,x_2] \oplus \bigoplus_{k\geq 1}x_1^{ \left[ \frac{k+1}{2} \right] }x_2t^k K[x_1,x_2]$.

We claim that $A(\Gamma)$ is not a finitely generated $S$-algebra. Indeed, if for each $k\geq 1$, we denote $u_k:=x_1^{ \left[ \frac{k+1}{2} \right] }x_2t^k$, one can easily check that each $u_k$ is not an element of $S[u_1,\ldots,u_{k-1}]$.
\end{exm}

Let $a\in\Nni$ and consider $\Gamma(a)$ the unique smallest multicomplex containing $a$. According to \cite[Corollary 9.8]{hp}, $\Gamma(a)$ is well defined, and moreover, $\Gamma(a)=\{b\in \Nni\;:\; b\leq a \}$. Let $\omega$ be a weight on $\Gamma(a)$, i.e. we give a positive integer $\omega_a$. Our next goal is to describe the vertex cover algebra $A(\Gamma(a),\omega)$. 

Note that $\mathcal M(\Gamma(a))=\{a\}$. A vector $a\in\Nn$ is a vertex cover of degree $k\in \mathbb N$ for $(\Gamma(a),\omega)$, if and only if $a(1)b(1)+a(2)b(2)+\cdots a(n)b(n)\geq \omega_a k$. If $a(j)=\infty$, for some $j\in [1,n]$, then any vector $b\in \Nn$ with $b(j)>0$, is a $k$-vertex cover for $(\Gamma(a),\omega)$ for any $k\geq 0$. If $a(j)=0$, for some $j\in [1,n]$, then $b(j)$ does not contribute to the sum $\sum_{i=1}^n a(i)b(i)$. We consider two extreme cases:

(1) $a=(\infty,\ldots,\infty)$. In this case, any $b\in \Nn\setminus \{(0,\ldots,0)\}$ is a $k$-vertex cover for $(\Gamma(a),\omega)$, for any $k\geq 0$. Also, $(0,\ldots,0)$ is a $0$-vertex cover, but is not a $k$-vertex cover for $k\geq 1$. It follows that $A((\Gamma(a),\omega)) = S\oplus t(x_1,\ldots,x_n)S[t] = S[t]$. 

(2) $a=(0,\ldots,0)$. In this case, $\Gamma(a)$ has no $k$-vertex covers for $k\geq 1$. Therefore, $A(\Gamma(a),\omega) = S$.

Assume $a\in\Nni$ is not in any of the above cases. Without losing generality, we can assume $a=(a(1),\ldots,a(r),\infty,\ldots,\infty)$, where $r\leq n$ is a positive integer. Indeed, in order to compute $A((\Gamma(a),\omega))$, we may permute the variables, and we reduce to this case. Assume $r>0$. We denote $\tilde{a}=(a(1),\ldots,a(r)) \in \mathbb N^{r}$ and we consider the weight $\tilde{\omega}$ on $\Gamma(\tilde{a})$, defined by $\tilde{\omega}(\tilde{a}):=\omega_a$. With these notations, we have the following lemma.

\begin{lema}
Let $k$ be a positive integer and $b\in \Nn$. Denote $\tilde{b}=(b(1),\ldots,b(r))\in \mathbb N^r$.

(i) If $\tilde{b}$ is a $k$-vertex cover for $(\Gamma(\tilde{a}),\tilde{\omega})$, then $b$ is a $k$-vertex cover for $(\Gamma(a),\omega)$.

(ii) If $b=(\tilde{b},0,\ldots,0)$ is a $k$-vertex cover for $(\Gamma(a),\omega)$, then $\tilde{b}$ is a $k$-vertex cover for $(\Gamma(\tilde{a}),\tilde{\omega})$.
\end{lema}

\begin{proof}
(i) Since $\tilde(b)$ is a $k$-vertex cover for $(\Gamma(\tilde{a}),\tilde{\omega})$, it follows that
$a(1)b(1)+\cdots+a(r)b(r)\geq k\cdot \tilde{\omega}_{\tilde{a}} = k\cdot \omega_a$. Therefore, $a(1)b(1)+\cdots+a(r)b(r)+\cdots+a(n)b(n)\geq k\cdot \omega_a$, and thus $b$ is a $k$-vertex cover for $(\Gamma(a),\omega)$.

(ii) the proof is similar to (i).
\end{proof}

\begin{prop}
$A(\Gamma(a),\omega) \cong S\oplus t (A(\Gamma(\tilde{a}),\tilde{\omega}) \oplus (x_{r+1},\ldots,x_n)S[t])$.
\end{prop}

\begin{proof}
Let $b\in \Nn$ and fix a positive integer $k$. Note that if $b(j)>0$ for some $j>r$, then $b$ is a $k$-vertex cover for $A(\Gamma(a),\omega)$. Indeed, in this case, $\sum_{i=1}^n a(i)b(i) \geq a(j)b(j) = \infty > k\omega_a$. On the other hand, according to the previous lemma, if $b(j)=0$ for all $j>r$, then $b$ is a $k$-vertex cover for $A(\Gamma(a),\omega)$ if and only if $\tilde{b}=(b(1),\ldots,b(r))$ is a $k$-vertex cover for $(\Gamma(\tilde{a}),\tilde{\omega})$. It follows that
$A(\Gamma(a),\omega)_k \cong  A(\Gamma(\tilde{a}),\tilde{\omega})_k \oplus (x_{r+1},\ldots,x_n)S$. Since 
$A(\Gamma(a),\omega)_0 = S$, we get the conclusion.
\end{proof}

The above proposition shows that we can reduce to the case when $a\in \Nn\setminus\{(0,\ldots,0)\}$. By reordering the variables, we can assume that $a=(a(1),\ldots,a(p),0,\ldots,0)$, where $p>0$ and $a(j)>0$ for any $j\leq p$. We denote $\bar{a}=(a(1),\ldots,a(p))\in \mathbb N^n$ and we consider the weight $\bar{\omega}$ on $\Gamma(\bar{a})$, defined by $\bar{\omega}(\bar{a}):=\omega_a$. With these notations, we have the following lemma.

\begin{lema}
Let $k$ be a positive integer and $b\in \Nn$. Denote $\bar{b}=(b(1),\ldots,b(p))\in \mathbb N^p$. Then $b$ is a $k$-vertex cover for $(\Gamma(a),\omega)$ if and only if $\bar{b}$ is a $k$-vertex cover for  $(\Gamma(\bar{a}),\bar{\omega})$.
\end{lema}

\begin{proof}
Indeed, since $a(p+1)=\cdots=a(n)=0$, we have $\sum_{i=1}^n a(i)b(i)=\sum_{i=1}^p a(i)b(i)$ and therefore we get the required conclusion.
\end{proof}

As a direct consequence of the previous lemma, we get the following proposition.

\begin{prop}
$A(\Gamma(a),\omega) = A(\Gamma(\bar{a}),\bar{\omega})[x_{p+1},\ldots,x_n]$.
\end{prop}

\begin{teor}
Suppose after renumbering that $a(i)\in \mathbb N\setminus\{0\}$ for $1\leq i\leq p$, $a(i) = 0$ for $p<i\leq r$ and $a(i)=\infty$ for $r<i\leq n$ and let $\bar{a} =(a(1),\ldots,a(p))$. Then:

\[ A(\Gamma(a);\omega)= S \oplus t(A(\Gamma(\bar{a}),\bar{\omega})[x_{p+1},\ldots,x_{r}]\oplus (x_{r+1},\ldots,x_n)  S[t]).\]

Moreover, $A(\Gamma(a);\omega)$ is a finitely generated $S$-algebra.
\end{teor}

\begin{proof}
The decomposition of $A(\Gamma(a);\omega)$ is a direct consequence of Proposition $1.3$ and Proposition $1.5$. For the second statement, using the above decomposition, it is enough to consider the case when 
$a\in \Nn$ such that $a(i)>0$ for all $i$. Let $k\geq (a(1)+\cdots a(n)+1)\cdot \omega_a$ be an integer and let $b\in \Nn$ be a $k$-vertex cover of $A(\Gamma(a);\omega)$, i.e. $\sum_{i=1}^n a(i)b(i)\geq k\cdot \omega_a$.

Since $k\geq (a(1)+\cdots a(n)+1)\cdot \omega_a$, it follows that the set $I:=\{i\in [1,n]:\; b(i)>\omega \}$ is nonempty. We define the vector $b'\in\Nn$, by $b'(i):=b(i)-\omega_a$ for $i\in I$ and $b'(i)=b(i)$ otherwise. Let $b'':=b-b'$, $k'=k-|I|$ and $k'':=|I|$. One can easily check that $b'$ is a $k'$-vertex cover and $b''$ is a $k''$-vertex cover. It follows that $A(\Gamma(a);\omega)$ is generated, as $S$-algebra, by the monomials from $A(\Gamma(a);\omega)_l$ for $1\leq l \leq (a(1)+\cdots a(n)+1)\cdot \omega_a - 1$.
\end{proof}

We end our paper with the following example.

\begin{exm}
Let $\mathbf{1}=(1,\ldots,1)\in \Nn$ and $\omega(\mathbf{1}):=\omega_{\mathbf{1}}$ a positive integer. A $k$-vertex cover for $(\Gamma(\mathbf{1}),\omega)$, is a vector $b\in \Nn$ such that $b(1)+\cdots+b(n)\geq k\cdot \omega_{\mathbf{1}}$. It follows that $A(\Gamma(\mathbf{1}),\omega)_k = (x_1,\ldots,x_n)^{k\cdot \omega_{\mathbf{1}}}$ and therefore
\[A(\Gamma(\mathbf{1}),\omega) = S\oplus(\bigoplus_{k\geq 1} (x_1,\ldots,x_n)^{k\cdot \omega_{\mathbf{1}}} t^k). \]
Note that $x_1^{\omega_{\mathbf{1}}}t,\ldots,x_n^{\omega_{\mathbf{1}}}t$ is a finite system of generators for $A(\Gamma(\mathbf{1}),\omega)$ as $S$-algebra.
\end{exm}

\vspace{2mm} \noindent {\footnotesize
\begin{minipage}[b]{15cm}
 Mircea Cimpoea\c s, Simion Stoilow Institute of Mathematics, Research unit 5, P.O.Box 1-764, Bucharest 014700, Romania\\
 E-mail: mircea.cimpoeas@imar.ro
\end{minipage}}
\end{document}